\documentclass[12pt]{amsart}

\usepackage{enumerate}
\usepackage{amsfonts}
\usepackage{amsthm}
\usepackage{amsmath}
\usepackage{amssymb}

\usepackage{t1enc}
\usepackage{fullpage}
\usepackage{dsfont}
\usepackage{graphicx}
\usepackage{mathrsfs}
\usepackage{mathptmx}
\usepackage[a4paper,left=2.5cm,right=2.5cm, footskip=30pt]{geometry}
\usepackage{lipsum}
\usepackage{geometry}

\newcommand\blfootnote[1]{%
  \begingroup
  \renewcommand\thefootnote{}\footnote{#1}%
  \addtocounter{footnote}{-1}%
  \endgroup
}

\newtheorem{theorem}{Theorem}[section]
\newtheorem{prop}[theorem]{Proposition}
\newtheorem{lemma}[theorem]{Lemma}
\newtheorem{defi}[theorem]{Definition}

\newtheorem{kov}[theorem]{Corollary}
\theoremstyle{definition}
\newtheorem{megj}[theorem]{Remark}

\begin{document}


\begin{abstract}
Let $X$ be a paracompact topological space and $Y$ be a Banach space. In this paper, we will characterize the Baire-1 functions $f:X\rightarrow{Y}$ by their graph: namely, we will show that $f$ is a Baire-1 function if and only if its graph $gr(f)$ is the intersection of a sequence $(G_n)_{n=1}^{\infty}$ of open sets in $X\times{Y}$ such that for all $x\in{X}$ and $n\in\mathbb{N}$ the vertical section of $G_n$ is a convex set, whose diameter tends to $0$ as $n\rightarrow\infty$. Afterwards, we will discuss a similar question concerning functions of higher Baire classes and formulate some generalized results in slightly different settings: for example we require the domain to be a metrized Suslin space, while the codomain is a separable Fr\'echet space. Finally, we will characterize the accumulation set of graphs of Baire-2 functions between certain spaces.
\end{abstract}

\title{Characterizations and Properties of Graphs of Baire Functions}
\author{Bal\'azs Maga}
\date{}
\maketitle

\section{Introduction}

In \cite{A}, S. J. Agronsky, J. G. Ceder and T. L. Pearson gave an equivalent definition of the real valued Baire class 1 functions defined on a metric space $X$ by characterizing their graph, which we will denote throughout this paper by $gr(f)$. In their article, Theorem 2.2 stated the following:

\begin{prop} Let $X$ be a metric space. Let us call an open set $G\subseteq{X\times\mathbb{R}}$ an open strip if for each $x\in{X}$ the intersection of $G_n$ and $\{x\}\times\mathbb{R}$ is an interval. Let $f:X\rightarrow\mathbb{R}$ be a function. It is Baire-1 if and only if there is a sequence $(G_n)_{n=1}^{\infty}$ of open strips such that $\cap_{n=1}^{\infty}G_n=gr(f)$. \end{prop}

In the case $X=[0,1]$, they gave a somewhat elementary proof. However, for the case when $X$ was an arbitrary metric space, they used Michael's Selection Theorem (see \cite{M}). This fact might lead us to the idea to regard the Baire-1 functions defined on a paracompact topological space $X$ with values from a Banach space $Y$ instead of $\mathbb{R}$ as this selection theorem holds in this more general situation. Furthermore, if $Y$ is a Banach space, we can easily find a natural counterpart of the notion of open strips in $X\times{Y}$:

\begin{defi} We say that an open set $G\subseteq{X\times{Y}}$ is an open strip if the vertical section $G(x)=G\cap\left({\{x\}\times{Y}}\right)$ is convex for all $x\in{X}$. \end{defi}

Using this definition, the following holds:

\begin{theorem} Let $f:X\rightarrow{Y}$ be a function where $X$ is a paracompact topological space and $Y$ is a Banach space. Then $f$ is Baire-1 if and only if there is a sequence $(G_n)_{n=1}^{\infty}$ of open strips such that $\cap_{n=1}^{\infty}G_n=gr(f)$ and $diam(G_n(x))\rightarrow 0$ for each $x\in{X}$ as $n$ tends to infinity. \end{theorem}

Concerning the graphs of Baire-$\alpha$ functions for $\alpha>1$ countable ordinals, we did not manage to prove a theorem in the setting of Theorem 1.3. Analogous results to the one achieved in the direction in which we assume $f$ is Baire-$\alpha$ can be obtained though rather easily in even more general settings, as it will be shown in Theorem 1.6. Namely, we can find some nice properties of the graph of a Baire-$\alpha$ function $f:X\rightarrow{Y}$ where $X$ is a topological space and $Y$ is a metric space, and then maybe we can achieve more specific ones if we require some conditions on $Y$: for example in Theorem 1.3 we restricted $Y$ to be a Banach space and we managed to prove a property of $gr(f)$ which does not even make sense if $Y$ is an arbitrary metric space. Following this line of thought, we can have a generalization of this direction of Theorem 1.3 in more general settings for any Baire class, but before we would formulate it, we would like to recall the following notation for higher Borel classes:

\begin{defi} A set $A$ is of additive class 1, ($A\in\Sigma_{1}$), if and only if it is open. For any countable ordinal greater than zero, $A$ is of multiplicative class $\alpha$, ($A\in\Pi_{\alpha}$), if and only if its complement is in $\Sigma_{\alpha}$. Finally, $A$ is of additive class $\alpha$, ($A\in\Sigma_{\alpha}$), for $\alpha>1$ if and only if there is a sequence of sets $A_1, A_2, ...$ such that each $A_i$ is in $\Pi_{\alpha_i}$ for some $\alpha_i<\alpha$ and $\bigcup_{i=1}^{\infty}A_i=A$. \end{defi}

It is useful to remark that the behaviour of the Borel hierarchy can be a bit chaotic in general topological spaces. To be more precise, we prefer if the higher Borel classes contain the lower ones, that is for $0<\beta<\alpha<\omega_1$, every set in $\Pi_\beta$ or $\Sigma_\beta$ is also in $\Pi_\alpha$ and $\Sigma_\alpha$. However, this property does not hold necessarily: for example if we regard the cofinite topology over any uncountable set, we can immediately see that none of the nontrivial open sets is in $\Sigma_2$. The following result is well-known and can be easily obtained by transfinite induction: if X has the property that any open set is in $\Sigma_2$ (or equivalently, any closed set is in $\Pi_2$) then every set in $\Pi_\beta$ or $\Sigma_\beta$ is also in $\Pi_\alpha$ and $\Sigma_\alpha$ for any $0<\beta<\alpha<\omega_1$. The spaces satisfying this requirement are called $G_\delta$ or perfect spaces and their defining property can be regarded as a separation axiom: the closed sets can be separated from their complements using only countably many open sets. It can be easily checked that all the metrizable spaces are perfect spaces, which is a fact we will use in this paper.

Let us recall that a topological vector space is a Fr\'echet space if it is locally convex and complete with a translation invariant metric. As the concept of convex sets exists in Fr\'echet spaces, we can similarly define open strips of $X\times{Y}$ if $Y$ is a Fr\'echet space instead of being a Banach space. Furthermore, we can generalize this definition to higher Borel classes:

\begin{defi} Let $X$ be a topological space and $Y$ be a Fr\'echet space. We say that a $\Sigma_{\alpha}$ set $S\subseteq{X\times{Y}}$ is a $\Sigma_{\alpha}$-strip if the vertical section $S(x)=S\cap\left({\{x\}\times{Y}}\right)$ is convex for all $x\in{X}$. \end{defi}

\begin{theorem} Let $f:X\rightarrow{Y}$ be a Baire-$\alpha$ function where $X$ is a topological space, $Y$ is a metric space and let $\alpha$ be a countable ordinal. Then there exists a sequence $(G_n)_{n=1}^{\infty}$ of $\Sigma_{\alpha}$ sets in $X\times{Y}$ such that $\cap_{n=1}^{\infty}G_n=gr(f)$ and $diam(G_n(x))\rightarrow 0$ for each $x\in{X}$ as $n$ tends to infinity. Furthermore, if $Y$ is a Fr\'echet space, these $\Sigma_\alpha$ sets can be chosen to be $\Sigma_\alpha$ strips. \end{theorem}

Proving the converse of this theorem in this rather general setting appeared to be much more difficult. The idea we may follow is similar to the one we will use in the proof of Theorem 1.3: we construct Baire-$\alpha_n$ functions for $\alpha_n<\alpha$ through $\Sigma_\alpha$ sets satisfying the conditions we found during the proof of the other direction, and with the help of these conditions we attempt to show that $f$ is the pointwise limit of these Baire-$\alpha_n$ functions. The essence of this concept is the construction of these functions which reduces our question to a selection problem concerning Baire functions of a given class (e.g. in Theorem 1.3 to the problem continuous selections). With the help of the Kuratowski--Ryll-Nardzewski Theorem about measurable selections (see \cite{KRN} and Proposition 2.2 in the next section) we will prove the following for functions defined on Suslin spaces, which are the continuous images of Polish spaces:

\begin{theorem} Let $f:X\rightarrow{Y}$ be a function where $X$ is a metrized Suslin space, $Y$ is a separable Fr\'echet space. Then $f$ is Baire-$\alpha$ for some successor countable ordinal $\alpha$ if and only if there is a nested sequence $(G_n)_{n=1}^{\infty}$ of $\Sigma_{\alpha}$ strips in $X\times{Y}$ such that
\begin{itemize}
\item $\cap_{n=1}^{\infty}G_n=gr(f)$,
\item $diam(G_n(x))\rightarrow 0$ for each $x\in{X}$ as $n$ tends to infinity,
\item the projection of $(X\times{U})\cap{G_n}$ to $X$ is in $\Sigma_\alpha$ for each open subset $U$ of $Y$.
\end{itemize}
\end{theorem}

\medskip

Finally, we focus on generalizing some of the results of \cite{MB}, in which we characterized the accumulation sets of graphs of Baire-2 functions $f:[0,1]\rightarrow{\mathbb{R}}$. In that paper, Theorem 4.1 and 4.2 stated the followings:

\begin{prop} Suppose $T\subseteq{[0,1]\times\mathbb{R}}$ and let us denote the set of accumulation points of some $gr(f)$ by $L_f$. There exists a bounded Baire-2 function satisfying $L_f=T$ if and only if $T$ is closed and $T\cap(\{x\}\times{\mathbb{R}})\neq\emptyset$ for each $x\in{[0,1]}$.

Furthermore, there exists a not necessarily bounded Baire-2 function satisfying $L_f=T$ if and only if $T$ is closed and there is a countable set $D\subseteq{[0,1]}$ such that $T\cap(\{x\}\times{\mathbb{R}})\neq\emptyset$ for each $x\in{[0,1]\setminus{D}}$. \end{prop}

The proofs appearing in that article are a bit complicated, however elementary. We will see that these problems also can be handled in a much more general setting using stronger tools as certain selection theorems. We recall that a space is $\sigma$-compact if it can be expressed as the countable union of compact sets.

\begin{theorem} Let $T\subseteq{X\times{Y}}$, where $X$ is a $\sigma$-compact metrizable Suslin space with no isolated points, and $Y$ is a compact Fr\'echet space. There exists a Baire-2 function $f$ satisfying $L_f=T$ if and only if $T$ is closed and $T\cap(\{x\}\times{Y})\neq\emptyset$ for each $x\in{X}$.

Furthermore, if $Y$ is $\sigma$-compact, but not compact, there exists a Baire-2 function $f$ satisfying $L_f=T$ if and only if $T$ is closed and there is a countable set $D\subseteq{X}$ such that $T\cap(\{x\}\times{Y})\neq\emptyset$ for each $x\in{X\setminus{D}}$.
\end{theorem}

\section{Preliminaries}

In order to prove our theorems, we need to recall a classical result about the relationship of Baire classes and Borel classes (see \cite{K}). As it is short and useful to prove, we will not omit the proof and formulate it as a proposition:

\begin{prop} Let $f:X\rightarrow{Y}$ be a Baire-$\alpha$ function where $X$ is a topological space, $Y$ is a metric space, and $\alpha$ is a countable ordinal. Then for any open set $G\subseteq{Y}$ the set $f^{-1}(G)\subseteq{X}$ is a $\Sigma_{\alpha+1}$ set, or in other words, $f$ is a Borel-$(\alpha+1)$ mapping. \end{prop}

\begin{proof} We proceed by transfinite induction. For $\alpha=0$ the proposition states that for continuous functions the inverse image of an open set is open which is true by definition. What remains to discuss is the inductive step. Let us assume $\alpha\geq{1}$ and we already know the statement for smaller ordinals, and let $(f_k)_{k=1}^{\infty}$ be a sequence of functions from lower Baire classes whose pointwise limit is $f$, namely let $f_k$ be Baire-$\alpha_k$ where $\alpha_k<\alpha$. If $\alpha$ is a successor ordinal, we might assume $\alpha_k=\alpha-1$. Let us denote the neighborhood of radius $\varepsilon>0$ of a closed set $F$ by $B(F,\varepsilon)$, which is clearly an open set. Then we may construct the following decomposition of $G$ into closed sets $(F_n)_{n=1}^{\infty}$:
\begin{displaymath}
G=\bigcup_{n=1}^{\infty}Y\setminus{B}\left(Y\setminus{G}, \frac{1}{n}\right) =\bigcup_{n=1}^{\infty}F_n .
\end{displaymath}
One can easily check that our decomposition implies that $f(x)=\lim_{k\rightarrow\infty}f_k(x)\in{G}$ holds if and only if there is an $n$ such that $f_k(x)\in{F_n}$ for all large enough $k$. Indeed, as $F_n$ is closed, if there is such an $n$, then the sequence $(f_k(x))$ cannot converge out of $F_n\subseteq{G}$ hence $f(x)\in{G}$. Conversely, if $f(x)\in{G}$, it has a neighborhood of radius $\varepsilon$ for suitable positive $\varepsilon$ in $G$. By convergence, for large enough $k$ the point $f_k(x)$ is in the neighborhood of $f(x)$ of radius $\frac{\varepsilon}{2}$, thus $f_k(x)\in{B}\left(Y\setminus{G}, \frac{\varepsilon}{2}\right)$ for large enough $k$. Choosing $n$ such that $\frac{1}{n}<\frac{\varepsilon}{2}$ gives us a suitable $n$ in our statement, thus it proves the other direction of our equivalence.

This equivalence yields the following equation:
\begin{displaymath}
f^{-1}(G)=\{x : f(x)\in{G}\}=\bigcup_{n=1}^{\infty}\bigcup_{m=1}^{\infty}\bigcap_{k=m}^{\infty}\left\{x : f_k(x)\in{F_n}\right\}.
\end{displaymath}
Now, for any set $\left\{x : f_k(x)\in{F_n}\right\}$ the inductive hypothesis can be used: $f_k(x)$ is Baire-$\alpha_k$ thus the inverse image of any open set is $\Sigma_{\alpha_k+1}$, hence the inverse image of the closed set $F_n$ is in $\Pi_{\alpha_k+1}$. Indeed, the inverse image of the complement is the complement of the inverse image, and the complement of $F_n$ is open, while the complement of its inverse image is in $\Sigma_{\alpha_k+1}$ whose complement is in $\Pi_{\alpha_k+1}$. Now if $\alpha$ is a successor ordinal, these sets in $\Pi_{\alpha_k+1}$ are in $\Pi_\alpha$ as $\alpha_k+1=\alpha$. Otherwise, if $\alpha$ is a limit ordinal the sets in $\Pi_{\alpha_k+1}$ are in $\Pi_\alpha$ by definition: the same unions can be regarded.  Hence if we take the intersection of sets of these type, for all $k\geq{m}$, we will still have a $\Pi_\alpha$ set for any ordinal. Finally if we take the countable union of such sets (that is, for all $n$ and $m$) we will obtain a $\Sigma_{\alpha+1}$ set as the inverse image of the open set $G$. \end{proof}

\bigskip

The proofs of Theorem 1.7 and 1.9 rests on the following corollary of the Kuratowski--Ryll-Nardzewski Theorem about measurable selections, which we already mentioned in the introduction:

\begin{prop} Let $X$ be a metric space and let $Y$ be a separable complete metric space. Assume $\alpha\geq{1}$ is a countable ordinal and let $\Psi:X\rightarrow{2^{Y}}$ be a multifunction with nonempty closed values such that $\Psi^{-1}(G)$ is in $\Sigma_{\alpha}$ for each open subset $G$ of $Y$. Then $\Psi$ admits a Borel class $\alpha$ selection, that is a mapping $f:X\rightarrow{Y}$ such that the inverse image of any open set of $Y$ is in $\Sigma_{\alpha}$.
\end{prop}

As we already remarked, we are interested in Baire selections. However, Baire classes of functions and Borel classes of mappings have a strong relationship. For example, as we have seen in Proposition 2.1, any Baire-$\alpha$ function defined on a topological space with values from a metric space is a Borel-$(\alpha+1)$ mapping. Nevertheless we have to be cautious since the converse does not hold in general: for instance if $X$ is a connected topological space with at least two points, and $Y$ is the two point discrete space $\{0,1\}$, then the characteristic function of a single point of $X$ is Borel-$2$, but not Baire-$1$, as all the continuous functions from $X$ to $Y$ are constants. As our aim is to use Proposition 2.2 in as general setting as it is possible, it would be beneficial to know some results concerning conditions yielding the equivalence of Baire-$\alpha$ functions and Borel-$(\alpha+1)$ mappings. We can recall a special form of Theorem 8 of \cite{H} (in that paper, every space is assumed to be perfect):

\begin{prop} Let $X$ be a perfect Suslin space and $Y$ be a metric space. If $X$ is metrizable and has topological dimension zero, or $Y$ is a locally convex topological linear space then the family of Baire-$\alpha$ functions coincides with the family of Borel-$(\alpha+1)$ mappings. \end{prop}

\begin{megj} In \cite{K}, \cite{KRN} and \cite{H}, and in several further articles and books other types of notation are used for Borel classes, causing a subtle ambiguity with our recent paper. In particular, in many papers the elements of $\Sigma_0$ are the open sets instead of the elements of $\Sigma_1$, and the higher Borel classes are defined from this starting point the same way we did in Definition 1.4. It is worth mentioning that this translation of the indices only leads to a difference in the case of finite ordinals as in the definition of $\Sigma_\omega$ we consider the same unions. \end{megj}

We wish to use Propositions 2.2 and 2.3 simultaneously, thus we have to restrict our observations to spaces satisfying the conditions of both. This explains why we stated Theorem 1.7 in the setting we did.

\medskip

For the proof of Theorem 1.9, we first generalize Lemma 3.1 of \cite{MB}. We can formulate an almost identical proposition and the proof is also verbatim:

\begin{lemma} Let $X$ and $Y$ be $\sigma$-compact metric spaces such that $X$ has no isolated points. For a given closed set $T\subseteq{X\times{Y}}$ there exists a countable set $A\subseteq{X}$ such that there is a function $f:A\rightarrow{Y}$ satisfying $L_f=T$. \end{lemma}

\begin{proof} The product space $X\times{Y}$ is also $\sigma$-compact, hence there is an increasing sequence of compact sets $(C_n)_{n=1}^{\infty}$ with limit $X\times{Y}$. Then $T_n=T\cap{C_n}$ is also a compact set. We will construct $A$ and $f$ by induction. Let us consider an open ball of radius one around each point of $T_1$. These balls give an open cover of the compact set $T_1$ hence it is possible to choose a finite cover. Let us take a point in each ball of the finite cover such that the $x$ coordinates of these points are pairwise different. As none of the points of $X$ is isolated, it is clearly possible. Denote the set of these points by $F_1$, and the set of their $x$ coordinates by $A_1$. In the following step, let us take open balls of radii $\frac{1}{2}$ around each point of $T_2$, choose a finite cover, and take points in each of these balls with pairwise different $x$ coordinates, which are also distinct from the points in $A_1$. Let us define $A_2$ and $F_2$ analogously, and continue this procedure: in the $n^{\text{th}}$ step regard the $\frac{1}{n}$-neighborhoods of the points of $T_n$, and define the finite sets $F_n$ and $A_n$ using these open balls. Now if we let $A=\bigcup_{n=1}^{\infty}A_n$ and $F=\cup_{n=1}^{\infty}{F_n}$, these are countable sets, and we may define $f$ to be the function that assigns to every $x\in{A}$ the $y$ coordinate of the chosen point in $F$ above $x$. The equality $L_f=T$ can easily be checked, as in [3]. 
\end{proof}

\section{Baire-1 functions}

Before we would start the proof of Theorem 1.3, a short remark should be mentioned. Theorem 1.3 gives almost the same characterization for the graphs of Baire-1 functions from $X$ to $Y$ as the one given in \cite{ACP} for the graphs of Baire-1 functions from $X$ to $\mathbb{R}$: in our statement, we can find the reasonable counterpart of the open strip condition of the real-valued case. However, besides that we drew up an additional limit condition concerning the diameter of the vertical sections, that has importance during the proof of the direction in which we show that if the graph has the given properties, then it is the graph of a Baire-1 function, and we work with the closure of the convex sections. What we would like to emphasize that this limit condition is vital and can be found implicitly in the more specific form of the theorem, too. We formulate the relevant fact as a proposition, since the author of this paper firmly believes this result has been published already but has yet to see a source:

\begin{prop} Let $E$ be a finite dimensional Banach space and $(C_n)_{n=1}^{\infty}$ is a nested sequence of closed convex sets such that $\bigcap_{n=1}^{\infty}C_n$ equals a point $p$. Then $diam(C_n)\rightarrow{0}$. \end{prop}

\begin{proof}

Proceeding towards a contradiction, let us assume $diam(C_n)>d>0$ for all $n\in\mathbb{N}$. Then $C_n$ must contain a point $p_n$ such that $\|p-p_n\|>\frac{d}{2}$ as it easily follows from the triangle inequality. As $C_n$ is convex and it contains $p$ and $p_n$, it also contains the $[p,p_n]$ segment, and on this segment a point $x_n$ satisfying $\|p-x_n\|=\frac{d}{2}$. Now the points $x_n$ all lie on the boundary of the ball with centre $p$ and radius $\frac{d}{2}$. By a consequence of Riesz's lemma, in our finite dimensional space this set is compact, hence the sequence $(x_n)_{n=1}^{\infty}$ has an accumulation point on this boundary. Let us denote it by $x$. As $C_n$ is closed for each $n\in{\mathbb{N}}$ and their sequence is nested, it implies $x\in{C_n}$. As a consequence, $x\in{\bigcap_{n=1}^{\infty}C_n}=\{p\}$, which is clearly a contradiction as the distance of $x$ and $p$ is $\frac{d}{2}$.
\end{proof}

This proposition implies that if we work with a finite dimensional Banach space, it is unnecessary to have the additional limit condition concerning the diameters of the vertical sections. However, one can easily construct counterexamples to Proposition 3.1 if we permit infinite dimensional spaces, for instance we can take the subspace spanned by a countable set of linearly independent vectors as $C_1$, and then reduce this subspace by removing the generators one by one.

\begin{proof}[Proof of Theorem 1.3] The proof is similar to the one given in \cite{ACP} for the more specific case, with some suitable modifications. First, let us assume that $f$ is Baire-$1$, hence there is a sequence of continuous functions $(f_n)_{n=1}^{\infty}$ with pointwise limit $f$. Let us notice that the set $\left\{x : \|f_n(x)-f(x)\|<\frac{1}{k}\right\}$ is in $\Sigma_2$. Indeed, if we let $g_n(x)=f_n(x)-f(x)$, it is also a Baire-$1$ function, and the set we are interested in is $g_n^{-1}\{y : \|y\|<\frac{1}{k}\}$, which is the inverse image of an open ball. Applying Proposition 2.1 yields that $\left\{x : \|f_n(x)-f(x)\|<\frac{1}{k}\right\}$ is in $\Sigma_2$ as we stated. As a consequence, it can be written as the countable union of closed sets $A(n,k,i)\subseteq{X}$:
\begin{displaymath}
\left\{x : \|f_n(x)-f(x)\|<\frac{1}{k}\right\}=\bigcup_{i=1}^{\infty}A(n,k,i).
\end{displaymath}
We will define the subsets $H(n,k,i)$ of $X\times{Y}$ as follows:
\begin{displaymath}
H(n,k,i)=\left\{(x,y) : x\in{A(n,k,i)}, \|y-f_n(x)\|\geq\frac{1}{k}\right\}.
\end{displaymath}
We show that $H(n,k,i)$ is closed. In order to prove it, let us write it as an intersection of two sets which are easier to handle:
\begin{displaymath}
H(n,k,i)=\left[A(n,k,i)\times{Y}\right]\cap\left\{(x,y):\|y-f_n(x)\|\geq\frac{1}{k}\right\}.
\end{displaymath}
The first one of these sets on the right hand side is clearly closed in $X\times{Y}$ as $A(n,k,i)$ was closed in $X$, hence it suffices to prove that the second set on the right hand side is also closed. Let us define the following function $h_n:X\times{Y}\rightarrow\mathbb{R_+}$, where $\mathbb{R_+}$ denotes the nonnegative halfline:
\begin{displaymath}
h_n(x,y)=\|y-f_n(x)\|.
\end{displaymath}
Our claim is that the continuity of $f_n$ implies the continuity of $h_n$. To prove this, we need to show that the inverse image of an open set $G\subseteq\mathbb{R_+}$ under $h_n$ is open in $X\times{Y}$. Thus let us assume $h_n(x_0,y_0)\in{G}$ for some $(x_0,y_0)\in{X\times{Y}}$, which yields for some $\varepsilon>0$ its neighborhood of radius $\varepsilon$ is the subset of $G$, that is $B\left(h_n(x_0,y_0),\varepsilon\right)\subseteq{G}$. We need that $h_n(x,y)$ is also in $G$ if $(x,y)$ is an element of a suitable neighborhood $U$ of $(x_0,y_0)$. We state this holds if we regard the following neighborhood:
\begin{displaymath}
U=f_n^{-1}\left(B\left(f_n(x_0),\frac{\varepsilon}{2}\right)\right)\times{B\left(y_0,\frac{\varepsilon}{2}\right)}
\end{displaymath}
By the continuity of $f_n$ it is indeed a neighborhood of $(x_0,y_0)$ as $f_n^{-1}\left(B\left(f_n(x_0),\frac{\varepsilon}{2}\right)\right)$ is an open subset of $X$. Furthermore, if $(x,y)\in{U}$, by the triangle inequality we have
\begin{displaymath}
\|y-f_n(x)\|\leq{\|y-y_0\|+\|y_0-f_n(x_0)\|+\|f_n(x_0)-f_n(x)\|}<\varepsilon+\|y_0-f_n(x_0)\|,
\end{displaymath}
and
\begin{displaymath}
\|y-f_n(x)\|\geq{-\|y-y_0\|+\|y_0-f_n(x_0)\|-\|f_n(x_0)-f_n(x)\|}>-\varepsilon+\|y_0-f_n(x_0)\|,
\end{displaymath}
which implies $h_n(x,y)\in{B\left(h_n(x_0,y_0),\varepsilon\right)\subseteq{G}}$.

\noindent Thus $h_n$ is continuous indeed, yielding $\left\{(x,y):\|y-f_n(x)\|\geq\frac{1}{k}\right\}=h_n^{-1}\left(\left[\frac{1}{k},\infty\right)\right)$ is a closed set of $X\times{Y}$. By our previous remarks it implies that $H(n,k,i)$ is also closed.

\medskip

The set of such sets $H(n,k,i)$ is countable thus we can take an enumeration $H_1$, $H_2$, ..., of them. Let us denote by $G_j^*$ the complement of $H_j$ in $X\times{Y}$, that is an open set. Furthermore, one can easily check that $G_j^*$ is an open strip, that is the $G_j^*(x)$ vertical section is convex for each $j\in\mathbb{N}$ and $x\in{X}$. Indeed, by the construction of $G_j^*$, this vertical section is either the complete space $Y$ or the ball of radius $\frac{1}{k}$ centered at $f_n(x)$ for some $k\in\mathbb{N}$ and $x\in{X}$. However, balls are convex in Banach spaces, hence $G_j^*$ is an open strip. It implies $G_j=\bigcap_{l=1}^{j}G_l^*$ is also an open strip. Furthermore, the sequence $(G_j)_{j=1}^{\infty}$ is nested and $diam(G_j(x))$ tends to $0$ for each $x\in{X}$. Indeed, when we constructed $G_j$, we took the intersection of the complements of some sets $H(n,k,i)$. A vertical section of this complement is either the entire $Y$ or a ball with diameter $\frac{2}{k}$. But as all $x\in{X}$ appears in $A(n,k,i)$ for any $k$, for some $i$ and large enough $n$, this implies that $diam(G_j(x))\leq\frac{2}{k}$ for large enough $j$. As a consequence, $diam(G_j(x))\rightarrow{0}$. Hence if we could verify that the intersection of the open strips $(G_j^*)_{j=1}^{\infty}$ equals $gr(f)$, that would conclude the proof. But the proof of this fact is quite straightforward, we can check two inclusions. First, $(x,f(x))\in{G_j^*}$ for any $j\in\mathbb{N}$ and $x\in{X}$, implying $gr(f)\subseteq\bigcap_{j=1}^{\infty}{G_j}$. In order to show this, let us recall that the complement of $G_j^*$ is $H_j=H(n,k,i)$ for some $n,k,i\in\mathbb{N}$. We need $(x,f(x))\notin{H(n,k,i)}$. Proceeding towards a contradiction, let us assume $(x,f(x))\in{H(n,k,i)}$, yielding $x\in{A(n,k,i)}$. Then by the definition of $A(n,k,i)$, the inequality $\|f_n(x)-f(x)\|<\frac{1}{k}$ holds. However, $\|f(x)-f_n(x)\|\geq\frac{1}{k}$ by the definition of $H(n,k,i)$, a contradiction. Thus $gr(f)\subseteq\bigcap_{j=1}^{\infty}{G_j}$. For the other inclusion, it suffices to prove that for any $x\in{X}$ and $y\in{Y}$ distinct from $f(x)$, we have $(x,y)\in{H(n,k,i)}$ for suitable $n,k,i\in\mathbb{N}$. In order to verify this, choose $k$ such that $\frac{1}{k}<\frac{\|y-f(x)\|}{2}$ and $n$ such that $\|f_n(x)-f(x)\|<\frac{1}{k}$. As $f$ is the pointwise limit of $(f_n)_{n=1}^{\infty}$, it is possible. Then by definition there exists $i$ such that $x\in{A(n,k,i)}$. Furthermore, $\|y-f_n(x)\|\geq\|y-f(x)\|-\|f_n(x)-f(x)\|>\frac{1}{k}$ by the triangle-inequality, implying $(x,y)\in{H(n,k,i)}$, which concludes the proof. Thus $\bigcap_{j=1}^{\infty}(G_j)=gr(f)$, we finished the proof of this direction.

For the other direction, let us assume $gr(f)=\bigcap_{j=1}^{\infty}{G_j}$ where for each $j$ the set $G_j$ is an open strip. We can also assume that their sequence is nested as the finite intersection of open sets is open and any intersection of convex sets is convex. Thus $G_{j+1}\subseteq{G_j}$ for any $j$. Let us define $F_j$ as it follows:
\begin{displaymath}
F_j=\bigcup_{x\in{X}}\overline{G_j(x)},
\end{displaymath}
where the overline means the closure. Hence $F_j$ stands for the closure by coordinates. Regard it as a multivalued function defined on $X$ with range $2^Y$, whose values are naturally the vertical sections of the set. Then this multivalued function has nonempty closed, convex values. Furthermore, we can easily show that $F_j$ is lower hemicontinuous: let us assume $V\cap{F_j(x)}$ is nonempty for some open set $V$ of $Y$ and $x\in{X}$. Since $F_j(x)$ is the closure of the open set $G_j(x)$, we have $V\cap{G_j(x)}$ is nonempty. Let $y\in{Y}$ be one of its elements. As $G_j$ is open, it contains a neighborhood of $(x,y)$. This neighborhood intersects $X\times\{y\}$ in a set whose projection to $X$ is open and suitable for us in the definition of lower hemicontinuity as one can easily check. Thus $F_j:X\rightarrow{2^Y}$ is a lower hemicontinuous function with nonempty closed, convex values. By the Michael selection theorem there exists a continuous selection $f_j:X\rightarrow{Y}$ in $F_j$. Furthermore, as the intersection of the sets $F_j(x)$ is only $f(x)$ and their diameter tends to $0$, we obtain $f_j(x)\rightarrow{f(x)}$. Hence $f$ is the pointwise limit of continuous functions, meaning $f$ is Baire-1. \end{proof}

\section{Higher Baire classes}

\begin{proof}[Proof of Theorem 1.6] The proof has a similar structure to the proof of Theorem 1.3, we just have to be more careful with the sets in higher Borel classes and make some slight, but necessary changes. As $f$ is Baire-$\alpha$, there is a sequence of functions $(f_n)_{n=1}^{\infty}$ with pointwise limit $f$, where $f_n$ is Baire-$\alpha_n$ for some $\alpha_n<\alpha$, and if $\alpha$ is a successor ordinal, we can assume $\alpha_n=\alpha-1$. Proposition 2.1 easily yields that the set $\left\{x : d_Y\left(f_n(x),f(x)\right)<\frac{1}{k}\right\}$ is in $\Sigma_{\alpha+1}$. To verify this, we show that if the functions $g_1,g_2: X \to Y$ are Baire-$\alpha$, then the function $\rho_{g_1,g_2}:X\to \mathbb{R}_+$ defined by $\rho_{g_1,g_2}(x)=d_Y\left(g_1(x),g_2(x)\right)$ is also Baire-$\alpha$. We proceed by transfinite induction: if $\alpha=0$, that is our functions are continuous, then our claim can be proven as the similar statement in the proof of Theorem 1.3. Furthermore, if we have $\alpha>0$, then $g_1$ is the pointwise limit of the functions $(g_{1,n})_{n=1}^{\infty}$ and $g_2$ is the pointwise limit of the functions $(g_{2,n})_{n=1}^{\infty}$, such that these functions are in lower Baire classes. Thus by the continuity of the metric $d_Y$, we have
\begin{displaymath}
\rho_{g_1,g_2}(x)=\lim_{n\to\infty}d_Y\left(g_{1,n}(x),g_{2,n}(x)\right)=\lim_{n \to\infty}\rho_{g_1,g_2,n}(x).
\end{displaymath}
However, the induction hypothesis easily yields that each of the functions $\rho_{g_1,g_2,n}$ are in lower Baire classes than Baire-$\alpha$. Thus $\rho_{g_1,g_2}$ is a Baire-$\alpha$ function, as we stated. As a consequence,
\begin{displaymath}
\left\{x : d_Y\left(f_n(x),f(x)\right)<\frac{1}{k}\right\}=\rho_{g_1,g_2}^{-1}\left(\left[0,\frac{1}{k}\right)\right)
\end{displaymath}
is in $\Sigma_{\alpha+1}$ by Proposition 2.1, as we consider the inverse image of an open set in $\mathbb{R}_+$ under a Baire-$\alpha$ function. Thus it can be written as the countable union of $\Pi_\alpha$ sets $A(n,k,i)\subseteq{Y}$:
\begin{displaymath}
\left\{x : d_Y\left(f_n(x),f(x)\right)<\frac{1}{k}\right\}=\bigcup_{i=1}^{\infty}A(n,k,i).
\end{displaymath}
We define the subsets $H(n,k,i)$ of $X\times{Y}$ as follows:
\begin{displaymath}
H(n,k,i)=\left\{(x,y) : x\in{A(n,k,i)}, d_Y\left(y,f_n(x)\right)\geq\frac{1}{k}\right\}.
\end{displaymath}

\noindent We state $H(n,k,i)$ is in $\Pi_\alpha$. The proof of this claim starts with the same reformulation, that is we write $H(n,k,i)$ as the intersection of two simpler sets:
\begin{displaymath}
H(n,k,i)=\left[A(n,k,i)\times{Y}\right]\cap\left\{(x,y):d_Y\left(y,f_n(x)\right)\geq\frac{1}{k}\right\}.
\end{displaymath}
The first one of these sets on the right hand side is clearly in $\Pi_\alpha$ in $X\times{Y}$ as $A(n,k,i)$ was in $\Pi_\alpha$ in $X$, hence it suffices to prove that the second set on the right hand side is also in $\Pi_\alpha$. Let us define the following function $h_n:X\times{Y}\rightarrow\mathbb{R_+}$:
\begin{displaymath}
h_n(x,y)=d_Y\left(y,f_n(x)\right).
\end{displaymath}
One can easily prove by transfinite induction on $\alpha_n$ that if $f_n$ is Baire-$\alpha_n$ then $h_n$ is also Baire-$\alpha_n$: the base case $\alpha_n=0$, where $f_n$ is continuous, can be verified exactly as we did it in the proof of Theorem 1.3, we only have to replace the norms of the differences in the inequalities with the respective distances. Now if $f_n$ is the pointwise limit of a sequence of functions $(\phi_{n,m})_{m=1}^{\infty}$ from lower Baire classes, then $h_n$ is the pointwise limit of the sequence of functions $d_Y\left(y,\phi_{n,m}(x)\right)_{m=1}^{\infty}$, and for these functions the inductive hypothesis can be used. Hence $h_n$ is Baire-$\alpha_n$, yielding $\left\{(x,y):d_Y\left(y,f_n(x)\right)\geq\frac{1}{k}\right\}$ is in $\Pi_{\alpha_n+1}$, and as a consequence, it is also in $\Pi_\alpha$, as we can separate the cases of successor and limit ordinals as in the proof of Proposition 2.1. Thus $H(n,k,i)$ is in $\Pi_\alpha$.

At this point, we can proceed exactly as we did in the proof of Theorem 1.3. We can take an enumeration $H_1, H_2, ...$ of the sets $H(n,k,i)$ and define $G_j$ as $(X\times{Y})\setminus\bigcup_{l=1}^{j}H_l$. Then these sets are in $\Sigma_\alpha$ and their intersection is $gr(f)$. Furthermore, if $Y$ is a Fr\'echet space, these sets are also $\Sigma_\alpha$ strips as balls are convex sets in Fr\'echet spaces. \end{proof}

\begin{proof}[Proof of Theorem 1.7]
For the direction in which we assume that $f$ is Baire-$\alpha$, we can refer to the proof of Theorem 1.6, the only detail we have to check that is the third condition is also satisfied. In general, let us denote the projection of a set $C\subseteq{X\times{Y}}$ to $X$ by $\pi(C)$, and let us denote the projection of $(X\times{U})\cap{G_n}$ to $X$ for the sake of simplicity by $\pi_n^*(U)$. Let us define the sets $A(n,k,i)$ and $H(n,k,i)$, and then the sequences $(H_j)_{j=1}^{\infty}$ and $(G_j)_{j=1}^{\infty}$ as we did in that proof. Namely, if $(f_n)_{n=1}^{\infty}$ is the sequence of functions from lower Baire classes with pointwise limit $f$, then
\begin{displaymath}
\left\{x : d_Y(f_n(x),f(x))<\frac{1}{k}\right\}=\bigcup_{i=1}^{\infty}A(n,k,i), \text{where } A(n,k,i)\in\Pi_\alpha,
\end{displaymath}
\begin{displaymath}
H(n,k,i)=\left[A(n,k,i)\times{Y}\right]\cap\left\{(x,y):d_Y(y,f_n(x))\geq\frac{1}{k}\right\},
\end{displaymath}
\noindent $(H_j)_{j=1}^{\infty}$ is the enumeration of these sets $H(n,k,i)$, and
\begin{displaymath}
G_j=(X\times{Y})\setminus\bigcup_{l=1}^{j}H_l=\bigcap_{l=1}^{j}(X\times{Y})\setminus{H_l}.
\end{displaymath}
\noindent Our goal is to prove that $\pi_j^*(U)$ is in $\Sigma_\alpha$ for each open subset $U$ of $Y$. Assume that $G_j$ can be decomposed as the following:
\begin{displaymath}
G_j=(X\times{Y})\setminus\bigcup_{l=1}^{j}H(n_l,k_l,i_l)=\bigcap_{l=1}^{j}(X\times{Y})\setminus{H(n_l,k_l,i_l)}.
\end{displaymath}
\noindent Now we can divide each $(X\times{Y})\setminus{H(n_l,k_l,i_l)}$ into two parts with disjoint projections to $X$:
\begin{displaymath}
(X\times{Y})\setminus{H(n_l,k_l,i_l)}=\left[(X\setminus{A(n_l,k_l,i_l}))\times{Y}\right] \cup \left[(A(n_l,k_l,i_l)\times{Y})\setminus{H(n_l,k_l,i_l)}\right]=V_{l,1} \cup V_{l,2},
\end{displaymath}
\noindent yielding
\begin{displaymath}
G_j=\bigcap_{l=1}^{j}(X\times{Y})\setminus{H(n_l,k_l,i_l)}=\bigcap_{l=1}^{j}(V_{l,1} \cup V_{l,2}).
\end{displaymath}
\noindent By distributivity, we can replace this intersection of unions by a union of intersections:
\begin{displaymath}
\bigcap_{l=1}^{j}(V_{l,1} \cup V_{l,2})=\bigcup_{(\theta_1,...\theta_j) \in \{1,2\}^j} \bigcap_{l=1}^{j}V_{l,\theta_l}.
\end{displaymath}
\noindent What is intriguing about this expression, that is the projections of the sets $\bigcap_{l=1}^{j}V_{l,\theta_l}$ to $X$ are clearly disjoint as two such intersection differs in at least one $\theta$-coordinate, and the projections $\pi(V_{l,1})$ and $\pi(V_{l,2})$ are disjoint. As a consequence, the projection of the union $$\bigcup_{(\theta_1,...\theta_j) \in \{1,2\}^j} \bigcap_{l=1}^{j}V_{l,\theta_l}$$ to $X$ equals the union of the projections, hence
\begin{equation}
\pi_j^*(U)=\pi\left((X\times{U})\cap{G_j}\right)=\bigcup_{(\theta_1,...\theta_j) \in \{1,2\}^j}\pi\left((X\times{U})\cap\bigcap_{l=1}^{j}V_{l,\theta_l}\right).
\end{equation}
\noindent We would like to show that this set is in $\Sigma_\alpha$. Let us consider one of these sets $$\pi\left((X\times{U})\cap\bigcap_{l=1}^{j}V_{l,\theta_l}\right)$$ and take a closer look at $\bigcap_{l=1}^{j}V_{l,\theta_l}$. Amongst these sets, certain ones are of the type $V_{l,1}$, others are of the type $V_{l,2}$. Let us denote the set of indices belonging to the first type by $J_1$, and the set of indices belonging to the second type by $J_2$, yielding
\begin{displaymath}
\pi\left((X\times{U})\cap\bigcap_{l=1}^{j}V_{l,\theta_l}\right)=\pi\left((X\times{U})\cap\bigcap_{l\in{J_1}}V_{l,1}\cap\bigcap_{l\in{J_2}}V_{l,2}\right).
\end{displaymath}
\noindent In this expression, $V_{l,1}=\left[(X\setminus{A(n_l,k_l,i_l)})\times{Y}\right]$ for $l\in{J_1}$, meaning $V_{l,1}$ contains the whole space $Y$ above $X\setminus{A(n_l,k_l,i_l)}$. As a consequence, one can easily verify that
\begin{equation}
\pi\left((X\times{U})\cap\bigcap_{l\in{J_1}}V_{l,1}\cap\bigcap_{l\in{J_2}}V_{l,2}\right)=\bigcap_{l\in{J_1}}(X\setminus{A(n_l,k_l,i_l)})\cap\pi\left((X\times{U})\cap\bigcap_{l\in{J_2}}V_{l,2}\right).
\end{equation}
\noindent Let us recall the definiton of $V_{l,2}$:
\begin{displaymath}
\pi\left((X\times{U})\cap\bigcap_{l\in{J_2}}V_{l,2}\right)=\left\{x: x\in\bigcap_{l\in{J_2}}{A}(n_l,k_l,i_l), U\cap\bigcap_{l\in{J_2}}B_Y\left(f_{n_l}(x),\frac{1}{k_l}\right)\neq\emptyset\right\}.
\end{displaymath}
\noindent Using these identities, we can reformulate (1), yielding $\pi_j^*(U)$ equals the following:
\begin{equation}
\bigcup_{J_1,J_2}\left(\bigcap_{l\in{J_1}}(X\setminus{A(n_l,k_l,i_l)})\cap\left\{x: x\in\bigcap_{l\in{J_2}}{A}(n_l,k_l,i_l), U\cap\bigcap_{l\in{J_2}}B_Y\left(f_{n_l}(x),\frac{1}{k_l}\right)\neq\emptyset\right\}\right).
\end{equation}
\noindent Now we will show that the condition $x\in\bigcap_{l\in{J_2}}{A}(n_l,k_l,i_l)$ might be omitted from this expression for each $J_1,J_2$ without changing the union. This omission extends each of the sets
\begin{displaymath}
\bigcap_{l\in{J_1}}(X\setminus{A(n_l,k_l,i_l)})\cap\left\{x: x\in\bigcap_{l\in{J_2}}{A}(n_l,k_l,i_l), U\cap\bigcap_{l\in{J_2}}B_Y\left(f_{n_l}(x),\frac{1}{k_l}\right)\neq\emptyset\right\}
\end{displaymath}
\noindent to
\begin{equation}
\bigcap_{l\in{J_1}}(X\setminus{A(n_l,k_l,i_l)})\cap\left\{x: U\cap\bigcap_{l\in{J_2}}B_Y\left(f_{n_l}(x),\frac{1}{k_l}\right)\neq\emptyset\right\},
\end{equation}
\noindent however, as we will show the increment is contained by other sets of the union in (3), yielding this union remains the same. Indeed, as $J$ runs over the subsets of $J_2$, the sets $$\bigcap_{l\in{J}}(X\setminus{A(n_l,k_l,i_l)})\cap\bigcap_{l\in{J_2\setminus{J}}}A(n_l,k_l,i_l)$$ give a natural partition of $X$. As a consequence, the set in (4) can be expressed as it follows, by taking the intersection with each of the elements of this partition and then forming their union:
\begin{equation}
\bigcup_{J\subseteq{J_2}}\left(\bigcap_{l\in{J\cup{J_1}}}(X\setminus{A(n_l,k_l,i_l)})\cap\bigcap_{l\in{J_2\setminus{J}}}A(n_l,k_l,i_l)\cap\left\{x: U\cap\bigcap_{l\in{J_2}}B_Y\left(f_{n_l}(x),\frac{1}{k_l}\right)\neq\emptyset\right\}\right).
\end{equation}
Taking the intersection of the sets $B_Y\left(f_{n_l}(x),\frac{1}{k_l}\right)$ only for $J_2\setminus{J}$ clearly extends this set, and $\bigcap_{l\in{J_2\setminus{J}}}A(n_l,k_l,i_l)$ can be moved inside $\left\{x: U\cap\bigcap_{l\in{J_2}}B_Y\left(f_{n_l}(x),\frac{1}{k_l}\right)\neq\emptyset\right\}$, yielding the set in (4) is contained by
\begin{displaymath}
\bigcup_{J\subseteq{J_2}}\left(\bigcap_{l\in{J_1\cup{J}}}(X\setminus{A(n_l,k_l,i_l)})\cap\left\{x : x\in\bigcap_{l\in{J_2\setminus{J}}}{A}(n_l,k_l,i_l), U\cap\bigcap_{l\in{J_2\setminus{J}}}B_Y\left(f_{n_l}(x),\frac{1}{k_l}\right)\neq\emptyset\right\}\right).
\end{displaymath}
\noindent Now we may notice that each of the unioned sets in this expression appears in the union in (3), which verifies our statement: we can make the omissions for any $J_1$ and $J_2$ without changing the union there. In other words, $\pi_j^*(U)$ is also the union of these modificated sets, that is
\begin{displaymath}
\pi_j^*(U)=\bigcup_{J_1,J_2}\left(\bigcap_{l\in{J_1}}(X\setminus{A(n_l,k_l,i_l)})\cap\left\{x: U\cap\bigcap_{l\in{J_2}}B_Y\left(f_{n_l}(x),\frac{1}{k_l}\right)\neq\emptyset\right\}\right).
\end{displaymath}

\noindent As it is a finite union, it suffices to prove about each of the unioned sets that they are in $\Sigma_\alpha$, that is
\begin{displaymath}
\bigcap_{l\in{J_1}}(X\setminus{A(n_l,k_l,i_l)})\cap\left\{x : U\cap\bigcap_{l\in{J_2}}B_Y\left(f_{n_l}(x),\frac{1}{k_l}\right)\neq\emptyset\right\}\in\Sigma_\alpha.
\end{displaymath}
\noindent The sets $X\setminus{A(n_l,k_l,i_l)}$ are also in $\Sigma_\alpha$, therefore it would be sufficient to prove the same about $\left\{x : U\cap\bigcap_{l\in{J_2}}B_Y\left(f_{n_l}(x),\frac{1}{k_l}\right)\neq\emptyset\right\}$. The intersection which we regard in this set is the intersection of a finite collection of open sets, hence it is also open. Furthermore, $U$ is separable as a subspace of the separable space $Y$. Thus it contains a countable dense set $\{u_1,u_2,...\}$. As a consequence, $U\cap\bigcap_{l\in{J_2}}B_Y\left(f_{n_l}(x),\frac{1}{k_l}\right)\neq\emptyset$ holds if and only if there exists some $u_t$ for $t\in{\mathbb{N}}$ such that $u_t\in\bigcap_{l\in{J_2}}B_Y\left(f_{n_l}(x),\frac{1}{k_l}\right)$, thus
\begin{displaymath}
\left\{x : U\cap\bigcap_{l\in{J_2}}B_Y\left(f_{n_l}(x),\frac{1}{k_l}\right)\neq\emptyset\right\}=\bigcup_{t=1}^{\infty}\left\{x :u_t\in\bigcap_{l\in{J_2}}B_Y\left(f_{n_l}(x),\frac{1}{k_l}\right)\right\}.
\end{displaymath}
\noindent For some $x\in{X}$, the relation $u_t\in{B_Y\left(f_{n_l}(x),\frac{1}{k_l}\right)}$ holds if and only if $f_{n_l}(x)\in{B_Y\left(u_t,\frac{1}{k_l}\right)}$ by symmetry. Hence
\begin{displaymath}
\bigcup_{t=1}^{\infty}\left\{x :u_t\in\bigcap_{l\in{J_2}}B_Y\left(f_{n_l}(x),\frac{1}{k_l}\right)\right\}=\bigcup_{t=1}^{\infty}\bigcap_{l\in{J_2}}\left\{x :f_{n_l}(x)\in{B_Y\left(u_t,\frac{1}{k_l}\right)}\right\}=\bigcup_{t=1}^{\infty}\bigcap_{l\in{J_2}}S(t,l).
\end{displaymath}
\noindent On the right hand side, each set $S(t,l)$ is the inverse image of an open set under $f_{n_l}$ which is Baire-$\alpha_{n_l}$, where $\alpha_{n_l}<\alpha$. Thus each $S(t,l)$ is in $\Sigma_\alpha$ by Proposition 2.1 as $X\times{Y}$ is metrizable, yielding that it is perfect. Hence if we take the finite intersection for $l\in{J_2}$ and then the countable union for $t=1,2,...$, we will still have a set in $\Sigma_\alpha$ and as we have already seen it concludes the proof of the first direction.

\medskip

For the other direction, let us assume $gr(f)=\bigcap_{j=1}^{\infty}{G_j}$ where the set $G_j$ is in $\Sigma_\alpha$ for each $j$, their sequence is nested, and they satisfy the three conditions of the theorem. Let us define $F_j$ as it follows:
\begin{displaymath}
F_j=\bigcup_{x\in{X}}\overline{G_j(x)},
\end{displaymath}
thus $F_j$ is the closure by coordinates. If we regard it as a multivalued function defined on $X$ with range $2^Y$, whose values are naturally the vertical sections of the set, we can easily verify that it satisfies the conditions of Proposition 2.2. Indeed, it has clearly nonempty, closed values, and as the projection of $(X\times{U})\cap{G_n}$ to $X$ is in $\Sigma_\alpha$ for each open subset $U$ of $Y$, the inverse image $F_j^{-1}(U)$ is in $\Sigma_\alpha$ for the open subsets of $Y$. Hence $F_j$ has a Borel-$\alpha$ selection $f_j$. As $\alpha$ is a successor ordinal, $\alpha-1$ makes sense and Proposition 2.3 can be applied, yielding $f_j$ is Baire-$(\alpha-1)$. The conclusion is the same as it was in the proof of Theorem 1.3: as the intersection of the sets $F_j(x)$ is only $\{f(x)\}$ and their diameter tends to $0$, $f_j(x)\rightarrow{f(x)}$ must hold, and as a consequence, $f$ is the pointwise limit of Baire-$(\alpha-1)$ functions, meaning $f$ is Baire-$\alpha$. \end{proof}

\section{Accumulation points of graphs}

Before we start the proof of Theorem 1.9, we would like to remark that the conditions concerning $T$ are clearly necessary, even if we do not require $f$ to be Baire-2:

\begin{prop} In the setting of the first case of Theorem 1.9, if a subset $T$ of $X\times{Y}$ equals $L_f$ for a function $f:X\rightarrow{Y}$, then $T$ is closed and $T\cap(\{x\}\times{Y})\neq\emptyset$ for each $x\in{X}$.

Furthermore, in the setting of the second case of Theorem 1.9, if a subset $T$ of $X\times{Y}$ equals $L_f$ for a function $f:X\rightarrow{Y}$, then $T$ is closed and there is a countable set $D\subseteq{X}$ such that $T\cap(\{x\}\times{Y})\neq\emptyset$ for each $x\in{X\setminus{D}}$. \end{prop}

\begin{proof} As $L_f$ is the set of the accumulation points of $gr(f)$, it must be closed in both cases. On the other hand, in the first case, if we consider any $x\in{X}$, by our conditions there is a sequence $(x_n)_{n=1}^{\infty}$ with elements from $X$ distinct from $x$ and with limit $x$. Thus by the compactness of $Y$, for any $f$ the sequence $\left(f\left(x_n\right)\right)_{n=1}^{\infty}$ has a limit point, implying the sequence $\left(x_n,f\left(x_n\right)\right)_{n=1}^{\infty}$ has a limit point in $\{x\}\times{Y}$. Hence if $T=L_f$, the set $T$ has to intersect any vertical line in the first case.

In the second case, proceeding towards a contradiction, let us assume the set $D$ of points in $X$ satisfying $T\cap(\{x\}\times{Y})=\emptyset$ is uncountable and there exists a function $f:X\rightarrow{Y}$ for which $T=L_f$ holds. As both $X$ and $Y$ is $\sigma$-compact, it implies the existence of compact sets $C_X\subseteq{X}$ and $C_Y\subseteq{Y}$ such that $C_X\cap{D}$ is uncountable and the cardinality of $D^*=\{x : x\in{C_X\cap{D}}, f(x)\in{C_Y}\}$ is also uncountable. Thus by the separability of $X$, the set $D^*$ contains one of its accumulation points, $d$. Therefore there exists a sequence $(d_i)$ in $D^*$, ($d_i\neq{d}$) with limit $d$. Since all the elements of the sequence $(f(d_i))$ are in the compact set $C_Y$, it has a convergent subsequence, therefore $L_f(d)$ cannot be empty, while $T(d)$ is, a contradiction.  \end{proof}

However, Theorem 1.9 states for such a set $T$ we have a Baire-2 function satisfying $L_f=T$, yielding the following:

\begin{kov} In the setting of any case of Theorem 1.9, if a subset $T$ of $X\times{Y}$ equals $L_f$ for a function $f:X\rightarrow{Y}$, then there exists a Baire-2 function such that $L_f=T$. \end{kov}

\begin{proof}[Proof of Theorem 1.9] Let us regard the first case. Consider a metric on $X$. By Lemma 2.5, there exists a countable set $A\subseteq{X}$ and a function $f_0:A\rightarrow{Y}$ satisfying $L_{f_0}=T$. We wish to extend this function to $f:X\rightarrow{Y}$ such that $f$ is Baire-2 without making $L_f$ larger. In order to do this, define a multifunction $F:X\rightarrow{2^{Y}}$ the following way:
\[
 F(x) =
  \begin{cases} 
      \hfill \{f_0(x)\}    \hfill & \text{ if } x\in{A} \\
      \hfill T\cap(\{x\}\times{Y}) \hfill & \text{ if } x\in{X\setminus{A}}. \\
  \end{cases}
\]
As $T$ is closed and its vertical sections are nonempty, $F$ has nonempty closed values. Furthermore, $F^{-1}(G)$ is in $\Sigma_{3}$ for each open subset $G$ of $Y$. Indeed, $T\cap(X\times{G})$ is a set in $\Sigma_{2}$. Next we show that $\pi(T\cap(X\times{G}))$ is also in $\Sigma_{2}$. Let us recall that as $X\times{Y}$ is $\sigma$-compact, any closed set is the union of countably many compact sets, implying any set in $\Sigma_{2}$ is also the union of countably many compact sets. However, the projection of a compact set is obviously compact, thus closed. Hence $\pi(T\cap(X\times{G}))$ is in $\Sigma_{2}$ as we stated. Furthermore, one can easily verify that $F^{-1}(G)$ and $\pi(T\cap(X\times{G}))$ can differ only in the points of $A$, because if we regard $T$ as a multifunction whose values are its vertical sections, $T$ and $F$ differ only in $A$. Thus $F^{-1}(G)$ differs only in a countable set from a set in $\Sigma_{2}$, yielding it is in $\Sigma_{3}$: indeed, a countable set is always in $\Sigma_{2}$, thus if we add a countable set to $\pi(T\cap(X\times{G}))$ we obtain another set in $\Sigma_{2}$, while removing a countable set is equivalent to intersecting with its complement, which is in $\Pi_{2}$. Hence the set we are interested in is the intersection of a set in $\Sigma_{2}$ and a set in $\Pi_{2}$, which are both in $\Sigma_{3}$ as $X$ is metrizable. As a consequence, the intersection is also in $\Sigma_{3}$, as we stated. Hence $F$ satisfies all the conditions of Proposition 2.2, yielding it admits a Borel-$3$ selection $f$. This function $f$ is also Baire-$2$ since the conditions of Theorem 1.9 satisfy the conditions of Proposition 2.3. What remains to show that is $L_f=T$. We have already seen $T\subseteq{L_f}$ as $T={L_{f_0}}$ by the construction of $f_0$ and $L_{f_0}\subseteq{L_f}$ clearly holds. For the other inclusion, we only have to verify that there is no sequence in $gr(f)$ with limit outside of $T$. However, in that case there would be such a sequence in $gr\left(f_0\right)$ as every point of $gr(f)$ is in the closed set $T$, except for the ones in $gr\left(f_0\right)$. Nevertheless that would imply $L_{f_0}$ is already larger than $T$, which is a contradiction. 

\medskip

In the second case, we can proceed almost the same way. Let us define $f_0$ on a countable set $A$ provided by Lemma 2.5. As $Y$ is not compact, there exists a sequence $(y_1, y_2, ...)$ in $Y$ without any accumulation point. Furthermore, as $D\setminus{A}$ is a countable set, we can enumerate its elements, possibly finitely: $(d_1, d_2, ...)$. Let us define the multifunction $F:X\rightarrow{2^{Y}}$ as it follows:

\[
 F(x) =
  \begin{cases} 
      \hfill \{f_0(x)\}    \hfill & \text{ if } x\in{A} \\
      \hfill \{y_i\}       \hfill & \text{ if } x=d_i \\
      \hfill T\cap(\{x\}\times{Y}) \hfill & \text{ if } x\in{X\setminus(A\cup{D})}. \\
  \end{cases}
\]

The steps of the previous case can be repeated to show that we can apply Proposition 2.2 to $F$ without any difficulty, yielding the existence of a Borel-$3$ selection $f$, which is also Baire-$2$ by Proposition 2.3. What is a difference from the previous case, that in the proof of $L_f=T$ we have to take into account those sequences of points of $gr(f)$ which contain infinitely many points above $D\setminus{A}$. However, as the sequence $(y_1, y_2, ...)$ has no accumulation point in $Y$, such a sequence cannot have an accumulation point in $X\times{Y}$, thus $L_f=T$, indeed. \end{proof}

\subsection*{Acknowledgements} 

I am grateful to Zolt\'an Buczolich for his valuable remarks and for the time he spent with proofreading this paper which helped to improve the quality of the exposition.

\end{document}